\documentclass[reqno]{amsart}
\usepackage{graphicx}
\usepackage{amssymb}
\usepackage{amsfonts}
\usepackage{indentfirst}
\usepackage{amssymb,amsmath,amsthm}
\usepackage{latexsym,bm}
\usepackage{graphicx}
\usepackage{times}
\usepackage{indentfirst}

 \usepackage{pst-grad} 
 
 \usepackage{pst-plot} 

\usepackage[curve]{xypic}

\usepackage[colorlinks,linkcolor=black,anchorcolor=blue,citecolor=blue]{hyperref}

\newtheorem{theorem}{Theorem}[section]

\newtheorem{lemma}[theorem]{Lemma}
\newtheorem{corollary}[theorem]{Corollary}

\newtheorem{proposition}[theorem]{Proposition}
\newtheorem{example}[theorem]{Example}
\def\r#1{\uppercase\expandafter{\romannumeral#1}}

\numberwithin{equation}{section}

\def\a{\alpha}
\def\b{\beta}
\def\Qp{\mathbb{Q}_{p}}
\def\Zp{\mathbb{Z}_{p}}

\def\Z{\mathbb{Z}}

\def\P{\mathbb{P}^{1}(\mathbb{Q}_p)}
\def\PF{\mathbb{P}^{1}(\mathbb{F}_p)}

\def\Q{\mathbb{Q}}

\def\a{\alpha}
\def\b{\beta}

\begin{document}
\title{Minimality of $p$-adic rational maps with good reduction}

\author{Aihua Fan}
\address{School of Mathematics and Statistics, Central China Normal University, 430079, Wuhan, China \& LAMFA, UMR 7352, CNRS,
Universit\'e de Picardie Jules Verne, 33, Rue Saint Leu, 80039
Amiens Cedex 1, France }
\email{ai-hua.fan@u-picardie.fr}

\author{Shilei Fan}
\address{School of Mathematics and Statistics, Central China Normal University, 430079, Wuhan, China}
\email{slfan@mail.ccnu.edu.cn}

\author{Lingmin Liao}
\address{LAMA, UMR 8050, CNRS,
Universit\'e Paris-Est Cr\'eteil, 61 Avenue du
G\'en\'eral de Gaulle, 94010 Cr\'eteil Cedex, France}
\email{lingmin.liao@u-pec.fr}

\author{Yuefei Wang}
\address{Institute of  Mathematics, AMSS,
Chinese Academy of Sciences, 100190 Beijing, China.}
\email{wangyf@math.ac.cn}

\thanks{A. H. FAN was supported by NSF of China (Grant No. 11471132); S. L. FAN was supported by NSF of China (Grant No.s 11401236 and 11611530542); L. M. LIAO was  supported by 12R03191A-MUTADIS; Y. F. WANG was supported by NSF of China (Grant No. 11231009).}

\begin{abstract}
A rational map with good reduction in the field $\mathbb{Q}_p$ of $p$-adic numbers defines a $1$-Lipschitz dynamical system on the projective line $\mathbb{P}^1(\mathbb{Q}_p)$ over  $\mathbb{Q}_p$. The dynamical structure of such a system is completely described by a minimal decomposition. That is to say, $\mathbb{P}^1(\mathbb{Q}_p)$ is decomposed into three parts:  finitely many periodic orbits; finite or countably many minimal subsystems each consisting of a finite union of balls; and the attracting basins of  periodic orbits and minimal subsystems.
For any prime $p$, a criterion of minimality for  rational maps with good reduction is obtained. When $p=2$, a condition in terms of the coefficients of the rational map is proved to be necessary for the map being minimal and having good reduction, and sufficient for the map being minimal and $1$-Lipschitz.
  It is also proved that a rational map having good reduction of degree $2$, $3$ and $4$ can never be minimal on the whole space $\mathbb{P}^1(\mathbb{Q}_2)$.


\end{abstract}
\maketitle
\section{Introduction}\label{int}

The present study contributes to the theory of $p$-adic dynamical systems which has recently been intensively and widely developed. For this development, one can consult the books \cite{Anashin-Khrennikov-AAD,Baker-Rumely-book,SilvermanGTM241} and their bibliographies therein.  

For a prime number $p$,  let $\mathbb{Q}_p$ be the field of $p$-adic numbers and  $\Zp$ be the ring of integers in $\Qp$.
Ergodic theory on the ring $\mathbb{Z}_p$ 
  is extremely important for applications to automata theory, computer science and cryptology, especially in connection with pseudorandom numbers and uniform distribution of sequences. The minimality, or equivalently the ergodicity
 with respect to the Haar measure, of non-expanding
dynamical systems on $\Zp$ are extensively studied in \cite{ Ana06,CFF09,CP11Ergodic,DP09,FLYZ07,FL11, FLSq, FLCh, GKL01, Jeong2013,Larin02, OZ75}, and so on.

The dynamical properties of the fixed points of the rational maps have been studied in the space $\mathbb{C}_p$ of $p$-adic complex numbers \cite{ARS13,KM06,MR04,Satt2015} and in the adelic space \cite{DKM07}. The Fatou and Julia theory of the rational maps on $\mathbb{C}_p$, and  on the Berkovich space over $\mathbb{C}_p$, are also developed 
\cite{Benedetto-Hyperbolic-maps,Baker-Rumely-book,Hsia-periodic-points-closure,Rivera-Letelier-Dynamique-rationnelles-corps-locaux,SilvermanGTM241}.  
However,  the global dynamical structure of rational maps on $\Q_p$ remains unclear, though the rational maps of degree one are totally characterized in \cite{FFLW2014} . 

Let $\phi\in \Qp(z)$ be a rational map of degree  $d\geq 2$. Then $\phi$ induces a dynamical system on the projective line $\P$ over $\Q_p$, denoted  by $(\P, \phi)$. Let $E\subset \P$ be a subset such that $\phi(E)\subset E$. Then restricted to $E$, $\phi$ defines a subsystem $(E, \phi |_E)$. The subsystem $(E, \phi |_E)$ is called {\it minimal} if for any point $x\in E$, the orbit of $x$ under $\phi$ is dense in $E$.
 In this article, we suppose that $\phi$ has good reduction (see the definition below).  
 The minimality of $(\P, \phi)$ and its subsystems will be fully investigated.  As we will see, any rational map  $\phi$ having good reduction is $1$-Lipschitz continuous on $\P$ which is equipped with its spherical metric. This suggests that  $(\P, \phi)$
shares many properties with the polynomial dynamics on $\Zp$,
which are $1$-Lipschtz with respect to the metric induced by the
$p$-adic absolute value.

Let $\overline{~\centerdot~}$ denote the \emph{reduction modulo }$p$  from $\Zp$ to $\Zp/p\Zp$ such that $a\mapsto \overline{a}$ with $a\equiv\overline{a} \ ({\rm mod} \ p)$.
For a polynomial $f(z)=\sum_{i=0}^{d}a_i z^i \in\Zp[z]$, the {\it reduction of $f$  modulo $p$} is defined as $$\overline{f}(z)=\sum_{i=0}^{d}\overline{a}_i z^i.$$

Note that $\phi$ can be written as a quotient of polynomials $f,g\in\Zp[z]$ having no common factors, such that at least one coefficient of $f$ or $g$ has absolute value $1$. Let $\mathbb{F}_{p}=\Z_p/p\Z_p$ be the finite field of $p$ elements. The {\it degree} of a rational map $\phi$, denoted by $\deg {\phi}$, is the maximum of the degrees of its denominator and numerator without common factors.
The {\it reduction of $\phi$  modulo $p$} is the rational function (of degree at most $\deg {\phi}$)
$$\overline{\phi}(z)=\frac{\overline{f}(z)}{\overline{g}(z)}\in  \mathbb{F}_{p}(z),$$  obtained by canceling common factors in the reductions $\overline{f}(z),\overline{g}(z)$. If $\deg\overline{\phi}=\deg \phi$, we say $\phi$ has \emph{good reduction}. If $\deg\overline{\phi}<\deg \phi$, we say $\phi$ has \emph{bad reduction}.

Assume that  $\phi$ has good reduction.  Let $\rho(\cdot,\cdot)$ be the\emph{ spherical metric} on $\P$ (see the definition in Section \ref{prel}). Then the map $\phi$ is 1-Lipschitz continuous (everywhere non-expanding \cite[p.59]{SilvermanGTM241}): $$\rho(\phi(P_1),\phi(P_2))\leq \rho(P_1,P_2) ~~~~~~~\quad\mbox{for all}~P_1,P_2\in \P.$$
 The 1-Lipschitz continuity of   $\phi$  leads us to investigate the minimality and minimal decomposition of the dynamical system $(\P,\phi)$, as in the case of polynomial dynamical systems on $\Zp$ studied in \cite{FL11}.

Let $\mathbb{P}^{1}(\mathbb{F}_p)$ be the projective line over $\mathbb{F}_p$. The reduction $\overline{\phi}$ induces a transformation from $\mathbb{P}^{1}(\mathbb{F}_p)$ into itself. We denote by  $\phi^{ k}$  the $k$-th iteration of $\phi$.
 A criterion of the minimality of the system $(\P,\phi)$ is given in the following theorem.
\begin{theorem}\label{minimal}
Let $\phi \in \Q_p(z)$ be a rational map of $\deg \phi \ge 2$ with good reduction.
Then the dynamical system $(\P,\phi)$ is minimal if and only if  the following conditions are satisfied:
\begin{itemize}
  \item[\rm{(1)}] The reduction $\overline{\phi}$ is transitive on  $\mathbb{P}^{1}(\mathbb{F}_p)$.
  \item[\rm{(2)}]  $(\phi^{(p+1)})^{\prime}(0)\equiv 1 \ ({\rm mod} \  \ p)$ and  $|\phi^{(p+1)}(0)|_p= 1/p$.
  \item[\rm{(3)}] For the case $p=2 \mbox{ or } 3$, additionally, $|\phi^{(p+1)p}(0)|_p= 1/p^2$.
\end{itemize}
\end{theorem}

Under the spherical
metric  $\rho(\cdot,\cdot)$,    $\P$    can be considered as an infinite symmetric tree (see Figures \ref{2adic} and \ref{3adic}). This tree is in fact an infinite $(p+1)$-Cayley tree. The centered vertex which is the root of the tree, is called {\em Gauss point}. Other vertices correspond to balls with radius strictly less than one (with respect to the spherical metric). The points in $\P$ are the boundary points of the tree.
A vertex is said to be \emph{at level $n$} ($n\geq 1$), if there are $n$ edges between the vertex and the Gauss point. Then for $n\geq 1$, there are $(p+1)p^{n-1}$ vertices at level $n$.  Since a rational map with good reduction is  $1$-Lipschitz continuous with respect to the spherical metric, it will induce an action on the tree under which the Gauss point is fixed and the sets of vertices at the same level are invariant.

\begin{figure}
  \centering
  \includegraphics[width=0.5\textwidth]{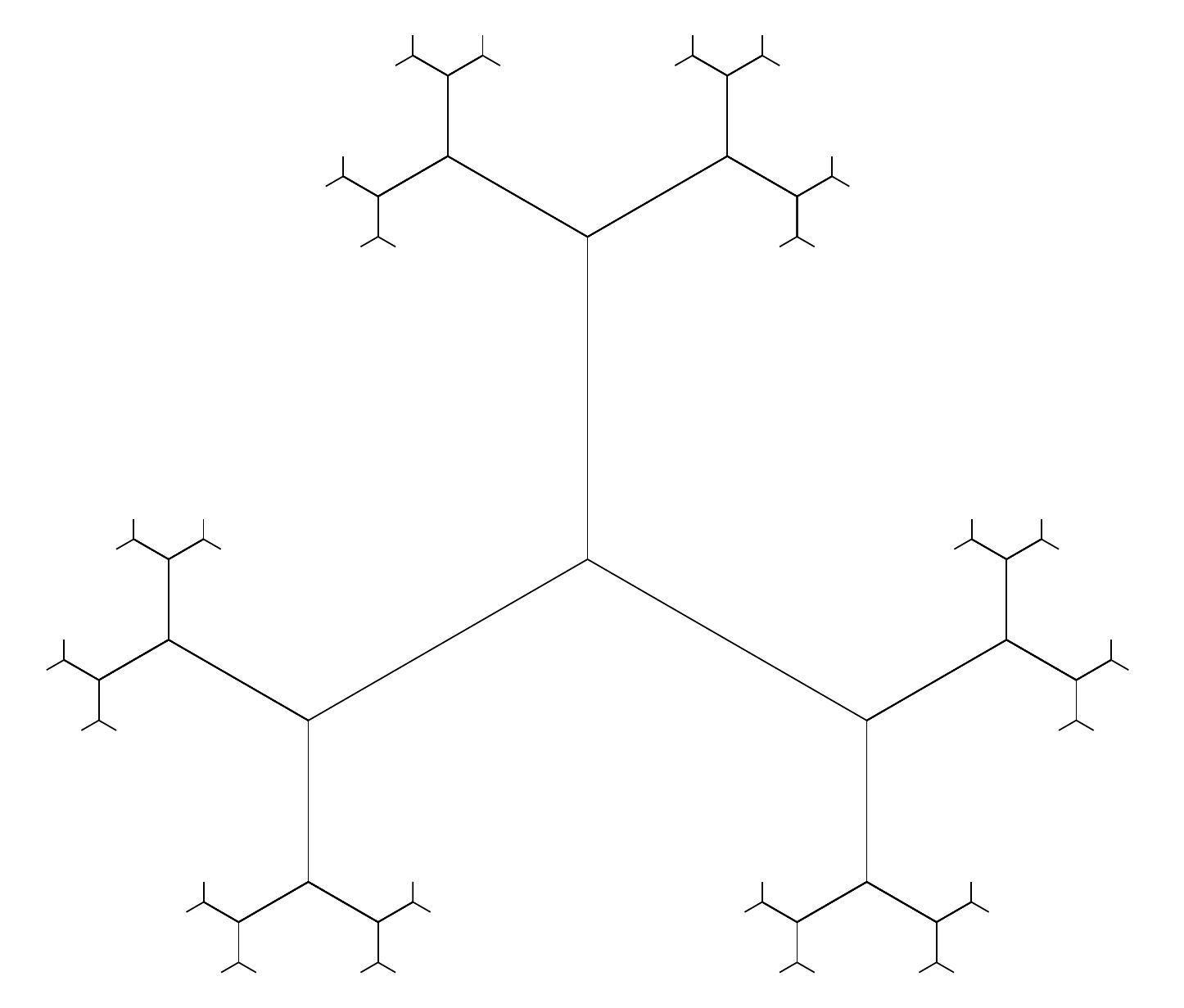}\\
  \caption{Tree structure of $\mathbb{P}^1(\Q_2)$. The points of $\mathbb{P}^1(\Q_2)$ are considered as the boundary  points of the infinite tree.}\label{2adic}
\end{figure}

Remark that if we consider $\phi$  as action on the tree,  the condition (1) in  Theorem \ref{minimal} means that  $\phi$ is transitive on the set of vertices at level 1, the conditions (1) and (2) imply  that  $\phi$ is transitive on the set of vertices at level $2$, while the condition (3) implies that $\phi$ is transitive on the set of vertices at level $3$ if conditions (1) and (2) are  also satisfied. We will see in Theorem \ref{minifinite} that the transitivity of $\phi$ on these finite levels is sufficient for the minimality of $\phi$ on the whole space. The conclusion of  Theorem \ref{minimal} can be  interpreted  as follows.

\emph{The dynamical system $(\P,\phi)$ is minimal if and only $\phi$ is transitive on the set of vertices at level $2$ for $p\geq 5$,  and at level $3$  for $p=2$ or $ 3$.}

\medskip
When the system $(\P, \phi)$ is not minimal, we describe the dynamical structure of the system by showing all its minimal subsystems.
\begin{theorem}\label{minmaldecomp}
Let $\phi \in \Q_p(z)$ be a rational map of $\deg \phi \ge 2$ with good reduction. We have the following decomposition
$$
     \P = \mathcal{P} \bigsqcup \mathcal{M} \bigsqcup \mathcal{B}
$$
where $ \mathcal{P} $ is the finite set consisting of all periodic points of
$\phi$, $ \mathcal{M} = \bigsqcup_i \mathcal{M} _i$ is the union of all (at most countably
many) clopen invariant sets such that each $\mathcal{M}_i$ is a finite union
of balls and each subsystem $\phi: \mathcal{M}_i \to \mathcal{M}_i$ is minimal, and 
points in $\mathcal{B}$ lie in the attracting basin of a periodic orbit or of
a minimal subsystem.  Moreover, the length of a periodic orbit has one of the following forms:
$$k \text{ ~~or ~~}k\ell, \quad \text{ if } p\geq 5,$$
$$k \text{ ~~or~~ }k\ell \text{~~ or~~ } kp, \quad \text{ if } p=2\text{ or }3,$$
where $1\leq k \leq p+1$  and $\ell\mid (p-1)$.
\end{theorem}
 The  decomposition in Theorem \ref{minmaldecomp} will be referred to as the {\em minimal
decomposition} of the system $\phi : \P \to \P$.  

We remark that the possible lengths of periodic  orbits were  investigated in 
\cite{Zievephd}, see also \cite[pp.62-64]{SilvermanGTM241}.

A finite periodic orbit of $\phi$ is by definition a minimal set.
But for the convenience of the presentation of the  paper, only the sets $\mathcal{M}_i$
in the above decomposition are called {\em minimal components}.
What kind of set can be a minimal component of a good reduction rational map?
In a recent work \cite{CFF09}, the authors showed that for any $1$-Lipschitz continuous map, a
minimal component  must be a Legendre set and that
any Legendre set is a minimal component of  some $1$-Lipschitz
continuous map. Recall that $E\subset \P$ is a {\it Legendre set} (see \cite[p.\,778]{CFF09}) if for every integer $n\geq 0$, every non-empty intersection of $E$ with a ball of radius $p^{-n}$ contains the same number of balls of radius $p^{-(n+1)}$.

We will show that for a  rational map with  good reduction, the minimal
components $\mathcal{M}_i$ are Legendre sets of special forms. Further, the dynamics of the minimal subsystems are adding machines. Let
$(p_s)_{s\ge 1}$ be a sequence of positive integers such that
$p_s|p_{s+1}$ for every $s\ge 1$. We denote by $\mathbb{Z}_{(p_s)}$
the inverse limit of $\mathbb{Z}/(p_s \mathbb{Z})$, called
{\em an odometer}. The sequence  $(p_s)_{s\ge 1}$ is called the {\em structure sequence}  of the   odometer. The map $z \to z+1$ on $\mathbb{Z}_{(p_s)}$ is the {\em adding machine} on
$\mathbb{Z}_{(p_s)}$. The structures of minimal components are determined as follows.

\begin{theorem}\label{structure-minimal}
Let $\phi \in \Q_p(z)$ be a rational map of $\deg \phi \ge 2$ with good reduction.  Assume that $E\subset \P$ is a minimal component  of $\phi$.
Then $\phi : E \to E$ is conjugate to the adding machine on an odometer
$\mathbb{Z}_{(p_s)}$, where  $$(p_s) = (k, k\ell, k \ell p, k\ell p^2,
\cdots)$$ with integers $k$ and $\ell$ satisfying that $1 \leq k\leq p+1$ and $\ell\mid
(p-1)$.
\end{theorem}
In applications, one usually would like to construct or determine a minimal rational (polynomial) map by its coefficients. However, this is far from easy, although a criterion of the minimality of rational maps with good reduction is given in Theorem \ref{minimal}. For rational maps of degree one, a complete description is given  in \cite{FFLW2014}.
 Here we give the description of  rational maps of degree at least two with good reduction, but only  for $p = 2$. 
 It seems to be a much more harder task for $p\geq 3$. 

For a rational map $\phi\in \Qp(z)$ of degree at least $2$, the number of periodic points of a fixed period must be finite. Hence, there exists a $z_0\in \Qp$ such that $z_0, \phi(z_0), \phi^{ 2}(z_0)$ are distinct.  Consider the linear fraction
 $$
 g(z)=\frac{(z-z_0)(\phi^{ 2}(z_0)-\phi(z_0))}{(z-\phi(z_0))(\phi^{ 2}(z_0)-z_0)}.
 $$
  Then $g(z_0)=0, g(\phi(z_0))=\infty$ and $g(\phi^{ 2}(z_0))=1$.  Let $\psi=g\circ \phi \circ g^{-1}$ be the conjugation of $\phi$ by $g$. Then we have  $\psi(0)=\infty$ and  $\psi(\infty)=1$ and the rational function  $\psi$ can be written as 
$$ \psi(z)=\frac{a_0+a_1z+\cdots +a_{d-1}z^{d-1}+z^d}{b_1z+\cdots + b_{d-1}z^{d-1}+ z^d}$$
where  $a_i,b_j \in \Q_p(0\leq i<d, 1\leq j<d)$ and $d\geq 2$ is the degree of $\phi$.
Therefore, without loss of generality, we can always assume that $\phi(0)=\infty$ and $\phi(\infty)=1$.

The following  theorem provides, in some sense,  a 
principle for constructing minimal rational maps on $\mathbb{P}^{1}(\Q_2)$.
For a  degree  $d\geq 2$ rational map  of form 
\begin{equation}\label{phi-def}
\phi(z)=\frac{a_0+a_1z+\cdots +a_{d-1}z^{d-1}+a_d z^d}{b_1z+\cdots + b_{d-1}z^{d-1}+ b_d z^d}
\end{equation}
  with  $a_i,b_j \in \Q_2$ and $a_d=b_d=1$, we set $A_\phi:=\sum_{i\geq 0}a_i$,
$B_\phi:=\sum_{j\geq 1}b_j$,  $A_{\phi,1}:=\sum_{i\geq 0}a_{2i+1}$, $A_{\phi,2}:=\sum_{i\geq 0}a_{4i+1}$ and $A_{\phi,3}:=\sum_{i\geq 0}a_{4i+3}$.

\begin{theorem}\label{p=2}
Let $\phi$ be defined as (\ref{phi-def}).  If  $\phi$  has good reduction  and   is minimal on $\mathbb{P}^{1}(\Q_2)$,    then 
\begin{equation}\label{equ1}
\begin{cases}
     a_i,b_j \in \Z_2, \hbox{~~~~~~~for $0\leq i \leq d-1$ and $1\leq j \leq d-1$},\\
     a_{0}\equiv 1 \ ({\rm mod}   \ 2),  \\
     B_{\phi} \equiv 1\ ({\rm mod}   \ 2), \\
     A_{\phi} \equiv 2 \ ({\rm mod}   \ 4), \\
     A_{\phi,1} \equiv 1 \ ({\rm mod}   \ 2),\\
     b_1 \equiv 1 \ ({\rm mod}   \ 2),\\
     a_{d-1}-b_{d-1} \equiv 1 \ ({\rm mod}   \ 2),\\
    a_0b_1(a_{d-1}-b_{d-1})(A_{\phi,2}-A_{\phi,3})B_\phi+\\
\quad  \quad 2(b_2-a_1+a_{d-2}-b_{d-2}+b_{d-1}+A_{\phi,3}) \equiv 1 \ ({\rm mod}   \ 4).
\end{cases}
\end{equation}
Conversely, the condition {\rm( \ref{equ1})} implies that $\phi$ is $1$-Lipschitz continuous and minimal on $\mathbb{P}^{1}(\Q_2)$. 
\end{theorem}
\begin{corollary}\label{non-exist}
Let $\phi \in \Q_2(z)$ be a rational map of degree  $2$, $3$ or $4$ with good reduction. Then the dynamical system $(\mathbb{P}^{1}(\Q_2),\phi)$ is not minimal.
\end{corollary}
 A complete characterization of minimal polynomial maps  on $\Zp$ in terms of their coefficients are given in \cite{Larin02} for $p = 2$ and in \cite{DP09} for $p=3$.  However, there is still no complete description for $p\geq 5$. On the other hand,  the characterizations of minimal Mahler  series  and van der Put series 
 by their coefficients are investigated in \cite{Ana94,Anashin-Khrennikov-AAD} and in \cite{Anashin-Khrennikov-Yurova,Jeong2013} respectively.

The above condition (\ref{equ1}) could be quickly checked by  computer, since there are only arithmetic operations `$+, -, \times$' in the  equations of (\ref{equ1}). 
By  condition   (\ref{equ1}) and a quick computer computation, we  obtain a series of   minimal rational maps  in $\Q_2(z)$ of degree  $4$  which is $1$-Lipschitz but do not have good reduction. See the table in Section \ref{deg4}.

The paper is organized as follows. In Section \ref{prel}, we give some preliminaries of $\P$,  including the spherical metric  and the tree structure of $\P$.
Section \ref{induc} is devoted to the induced dynamical systems on the vertices.  The proofs of Theorems \ref{minimal}--\ref{structure-minimal} are given in  Section \ref{pro}. In Section \ref{pequal2}, we characterize the minimal rational maps with good reduction  in terms of their  coefficients for the case $p=2$. Theorem \ref{p=2} and Corollary \ref{non-exist} will be proved in this section.   To illustrate our result, in the last section, we present  two rational maps with good reduction whose  exact minimal decompositions  are described when $p=3$.

\medskip

\section{Projective line } \label{prel}
 Any point in the projective line $\P$ may be given in homogeneous coordinates by a pair
$[x_1 : x_2]$
of points in $\Qp$ which are not both zero. Two such pairs are equal if they differ by an overall (nonzero) factor $\lambda \in \Q_p^*$:
$$[x_1 : x_2] = [\lambda x_1 : \lambda x_2].$$
The field $\Qp$ may be identified with the subset of $\P$ given by
$$\left\{[x : 1] \in \mathbb P^1(\Qp) \mid x \in \Qp\right\}.$$
This subset contains  all points  in $\P$ except one: the point  of infinity, which may be given as
$\infty = [1 : 0].$

The spherical metric defined on $\P$ is analogous to the standard spherical metric on the
Riemann sphere. If $P=[x_1,y_1]$ and $Q=[x_2,y_2]$ are two points in $\P$, we define
 $$\rho(P,Q)=\frac{|x_1y_2-x_2y_1|_p}{\max\{|x_1|_{p},|y_1|_{p}\}\max\{|x_2|_{p},|y_{2}|_{p}\}}$$
 or, viewing $\mathbb{P}^{1}(\Qp)$ as $\Qp\cup\{\infty\}$, for $z_1,z_2 \in \Qp\cup \{\infty\}$
 we define
 $$\rho(z_1,z_2)=\frac{|z_1-z_2|_{p}}{\max\{|z_1|_{p},1\}\max\{|z_2|_{p},1\}}  \qquad\mbox{if~}z_{1},z_{2}\in \Qp,$$
 and
 $$\rho(z,\infty)=\left\{
                    \begin{array}{ll}
                      1, & \mbox{if $|z|_{p}\leq 1$,} \\
                      1/|z|_{p}, & \mbox{if $|z|_{p}> 1$.}
                    \end{array}
                  \right.
 $$

Remark that the restriction of the spherical metric on the ring $\Zp:=\{x\in \Q_p, |x|\leq 1\}$ of $p$-adic integers is same as the metric induced by the absolute value $|\cdot|_p$.

A rational map $\phi \in \Qp(z)$ induces a transformation  on  $\P$.
Rational maps  are  always Lipschitz continuous on $\P$
with respect to the spherical metric (see \cite[Theorem 2.14.]{SilvermanGTM241}).
In particular,  rational maps  with  good reduction  are   $1$-Lipschitz  continuous.

\begin{lemma}[\cite{SilvermanGTM241}, p.59] Let $\phi\in \Qp(z)$ be a rational map with good reduction. Then the map $\phi: \P \mapsto \P$  is $1$-Lipschitz  continuous:
$$\rho(\phi(P),\phi(Q))\leq \rho(P,Q), ~~~~\forall P,Q \in \P. $$
\end{lemma}

\begin{figure}
  \centering
  \includegraphics[width=0.5\textwidth]{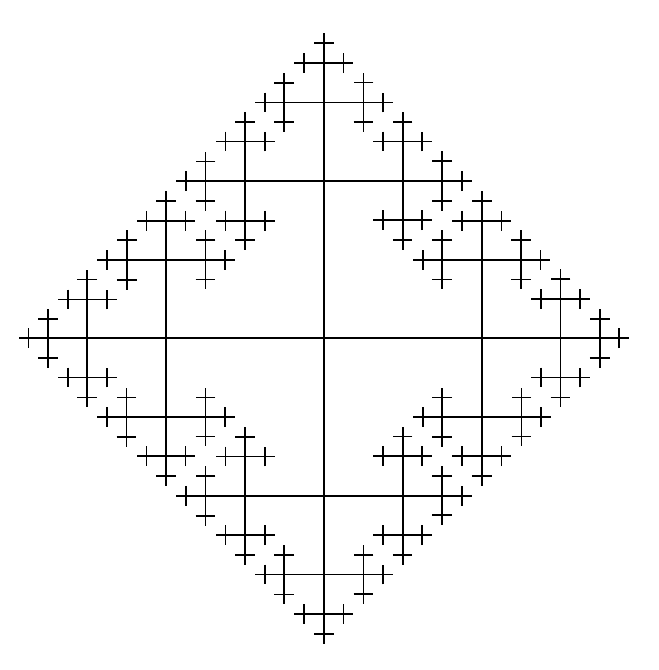}\\
  \caption{Tree structure of $\mathbb{P}^1(\Q_3)$. The points of $\mathbb{P}^1(\Q_3)$ are considered as the boundary points of the infinite tree.}\label{3adic}
\end{figure}

For $a\in \P$ and an integer $n\geq 1$, denote by
$$B_{n}(a):=\{x\in \P: \rho(x,a)\leq p^{-n}\}$$  the  \emph{ball} of radius $p^{-n}$ centered at $a$.
The projective line $\P$  consists of $p+1$ disjoint balls of radius $p^{-1}$,
$$\P= B_{1}(\infty)\bigsqcup\Big(\bigsqcup_{0\leq i \leq p-1} B_{1}(i)\Big).$$
For each integer $n\geq 1$, every ball of radius $p^{-n}$ consists of $p$ disjoint  balls of radius $p^{-n-1}$. For example, $B_1(\infty)$ consists of $B_2(1/p),B_2(2/p), \cdots, B_2((p-1)/p), B_2(\infty)$.

As mentioned in Section \ref{int}, the projective line $\P$ over $\Q_p$ could be considered as an infinite $(p+1)$-Caylay tree: the branch index of each vertex is
$p+1$, i.e. each vertex is an endpoint of $p+1$ edges. There is a centered vertex which is called Gauss point. Excluding Gauss point, each vertex corresponds to a ball in $\P$ of radius strictly less then one (with respect to the spherical metric).  The Gauss point can be considered as the ball of radius $1$  centered at the origin which is actually the whole space $\P$.  The set of edges of  the tree consists of the pairs of balls  $(B,B^{\prime})$ of radius $\leq 1$ such that
 $$B^{\prime}\subset B, \ r(B) =p \cdot r(B^{\prime})$$
  where  $r(B)$  and $r(B^{\prime})$  are the radii of balls $B$ and $B^{\prime}$ respectively.
  
The points in $\P$ are the boundary points of the tree. 
Remind that a vertex is said to be at  level $n (n\geq 1)$ if there are $n$  edges between Gauss point and the vertex. A vertex at level $n$ corresponds a ball of radius $p^{-n}$. The projective line $\P$ consists of $(p+1)p^{n-1}$ disjoint  balls of radius $p^{-n}$.
There are $(p+1)p^{n-1}$ vertices at level $n$. For example, see Figures \ref{2adic} and \ref{3adic} for the tree structures of $\mathbb{P}^{1}(\Q_2)$ and $\mathbb{P}^{1}(\Q_3)$.

\medskip

\section{Induced dynamics}\label{induc}
For each positive integer $n$, $\P$  consists of $(p+1)p^{n-1}$  disjoint balls of radius $p^{-n}$.
Denote by $\mathfrak{B}_n$ the set of the $(p+1)p^{n-1}$  disjoint balls of radius $p^{-n}$. In general, any $1$-Lipschitz continuous map  $\phi: \P \rightarrow \P$ induces a transformation  on   $\mathfrak{B}_n$. Denote by
$\phi_n$ the induced map of $\phi$ on $\mathfrak{B}_n$, i.e. $$\phi_n(B_{n}(x))=B_{n}(\phi(x)), \ \ \forall x\in \P.$$
As $\P$ is considered as  an infinite tree, for each positive integer $n$, there is a one-to-one correspondence between $\mathfrak{B}_n$  and the vertices of the tree at level $n$.    The $1$-Lipschitz continuous map  $\phi$ induces a transformation on the tree under which the set of vertices of  each level is  invariant.

Many properties of the dynamics $\phi$ on $\P$ are linked to those of $\phi_n$ on $\mathfrak{B}_n$.

\begin{proposition}\label{prop}
Let $\phi:\P \rightarrow \P$ be a $1$-Lipschitz continuous map. Then the system $(\P, \phi)$ is minimal if and only if
the finite system $(\mathfrak{B}_n,\phi_n)$ is minimal for all integers $n\geq 1$.
\end{proposition}
The  equivalence in Proposition \ref{prop} is similar  to that for $1$-Lipschitz continuous maps on the ring $\Z_p$  or on   a discrete valuation domain (see \cite[p.\,111]{Ana94}, \cite[Theorem 1.2]{Ana06} and \cite[Corollary 4]{CFF09}).

\begin{proof}The ``\,only if\," part of the statement is obvious, we prove only the  ``\,if\," part. Suppose that $(\P, \phi)$ is not minimal. Then there exist an open set $S$ and a point $z\in \P$
 such that $\phi^{ k}(z)\notin S $ for all integers $k\geq 1$. 
 Since the set $S$ is open, there exists a  ball $B_{n_0}(z_0) \subset S$ for some $z_0\in \P$ and some integer $n_0$ large enough.  Recall that $(\mathfrak{B}_n,\phi_n)$ is minimal for all integers $n>0$. So there  exists an integer $k_0$ such that $\phi_{n_0}^{k_0}(B_{n_0}(z))= B_{n_0}(z_0)$.  This  implies  $\rho(\phi^{k_0}(z),z_0)\leq p^{-n_0}$, which   contradicts the fact that $\phi^{k_0}(z)\notin S$.
\end{proof}

Recall that a rational map with good reduction is $1$-Lipschitz continuous.
For a rational map $\phi \in \Qp(z)$ with good reduction, to study the minimality of the system $(\P,\phi)$,
it suffices to  study the induced dynamics $(\mathfrak{B}_n,\phi_n)$ for all integers $n\geq 1 $.

\medskip
Now let us first study the dynamics $(\mathfrak{B}_1,\phi_1)$ at level $1$. 
Note that each point in $\PF$ corresponds to a ball in $\mathfrak{B}_1$. The induced map $\phi_1$ of $\phi$ can  be considered as the
reduction $\overline{\phi}$ acting on $\PF$. The map $\overline{\phi}: \PF \rightarrow \PF$ admits some periodic orbits.  A periodic orbit is call a  {\em cycle} of  $\overline{\phi}$.
The points outside any cycle will be  mapped into some cycle after several iterations.

Let $(\overline{x}_1,\cdots,\overline{x}_k)\subset \PF$ be a cycle of $\overline{\phi}$. To simplify notation, we consider  a transformation $f \in \mbox{PGL}_2(\Zp)$
with $f(x_1)=0$, where $x_1$ is a point in the ball corresponding to $\overline{x}_1$.
Replacing $x_1$ and $\phi$ with $0$ and $ f\circ \phi \circ f^{-1}$ respectively, we may assume that $\overline{x}_1=\overline{0}$.
Under this assumption, we will see in the following that $\phi^k$ acting on $p\Z_p$ is conjugate to a power series acting on $\Z_p$.

For two rational maps $\theta,\omega\in \Qp(z)$ with good reduction, the composition
$\theta\circ \omega$ has good reduction, and $\overline{\theta\circ \omega}=\overline{\theta}\circ\overline{\omega}$  (see \cite[p.59,  Theorem 2.18]{SilvermanGTM241}).
So $\phi^{ k}$ has good reduction and $\overline{0}$ is a fixed point of $\overline{\phi^{ k}}$. We write $\phi^{ k}$ in the form
$$\phi^{ k}(z)=\frac{a_0+a_1z+\cdots a_n z^n}{b_0+b_1z+\cdots b_n z^n}$$
with coefficients $a_0, \cdots a_d, b_0,\cdots  b_d \in \Zp$ and at least one coefficient in $\Zp \setminus p\Zp$. The fact that $\overline{0}$ is a fixed point of $\overline{\phi^{ k}}$ implies that
$\phi^{ k}(0)=a_0/b_0  \equiv 0 \ ({\rm mod} \ p)$.
Since $\phi^{ k}$ has good reduction, we have $a_0 \in p\Zp$ and $b_0\in \Zp\setminus p\Z_p$. Otherwise, the numerator and denominator of  the reduction   have the common factor $z$, which implies $\deg \overline{\phi^{ k}} <\deg  \phi^{ k}$.
Multiplying numerator and denominator by $b_0^{-1}$, we may thus write $\phi^{ k}$ in the
form $$\phi^{ k}(z)=\frac{a_0+a_1z+\cdots a_n z^n}{1+b_1z+\cdots b_n z^n}.$$
Then by the following Lemma \ref{lem-1-Lip}, $\phi^{ k}$ is $1$-Lipschitz from $p\Z_p$ to itself and
\[
\phi^k(z)=a_0+ \lambda z+ \lambda_2 z^2+ \lambda_3z^3+\cdots,
\] 
with $\lambda_i \in \Zp$ for all $i\geq 2$.

\begin{lemma}\label{lem-1-Lip}
Let $\phi\in \mathbb{Q}_p(z)$ be a rational map of the form
\[
\frac{a_0+a_1z+\cdots a_n z^n}{1+b_1z+\cdots b_n z^n}, \quad \text{with} \ a_j, b_j\in \mathbb{Z}_p.
\]
Then $\phi$ is $1$-Lipschitz on $p\mathbb{Z}_p$. Furthermore,
\begin{align}\label{inte-coe}
\phi(z)&=a_0+ \lambda z+ \lambda_2 z^2+ \lambda_3z^3+\cdots
\end{align}
with $\lambda_i \in \Zp$ for all $i\geq 2$. 
\end{lemma}

\begin{proof}
The Taylor expansion of $\phi$ around $z = 0$, which in this
case can be obtained by simple long division,  gives
$$\phi(z)=a_0+ \lambda z+ \frac{A(z)}{1+zB(z)}z^2$$
with
$A(z), B(z) \in  \Zp[z]$ and $\lambda = \phi^{\prime}(0)\in \Zp$.
Observe that on $p\Zp$, we have 
$$\frac{1}{1+zB(z)}=1-zB(z)+z^2(B(z))^2-z^3(B(z))^3+ \cdots.
$$
Hence  we can write $\phi(z)$ as the form (\ref{inte-coe}).
Obviously,  $\phi$  induces a $1$-Lipschitz map on $p\Zp$.
\end{proof}
\bigskip

Now we study the dynamical system $(p\Zp,\phi^{ k})$.
For simplicity, we write $\varphi$ instead of $\phi^{ k}$.
Note that every power series  $$\varphi(z)=\sum_{i=0}^{\infty}\lambda_i z^i \in \Zp[[z]] $$ converges on $p\Zp$. If $\lambda_0\in p\Zp$, then $\varphi(p\Zp) \subset p\Zp$ and $\varphi$  induces a $1$-Lipschitz map from $p\Zp$ to itself.

In the remainder of this section we assume $\varphi(z)=\sum_{i=0}^{\infty}\lambda_i z^i \in \Zp[[z]]$ with $\lambda_0 \in p\Zp$. The dynamical system 
$(p\Zp,\varphi)$ is conjugate to a system $(\Z_p,\chi)$ by the transformation $f(z)=z/p$, i.e.
\[\diagram
p\Zp\rto^{\varphi~~} \dto_{ z/p} &p\Z_p  \dto^{z/p}\\
\Z_p\rto_{\chi~~} & \Z_p,
\enddiagram
\]
where $\chi(z)=\sum_{i=0}^{\infty}p^{i-1}\lambda_i z^i  \in \Zp[[z]]$   converges on $\Z_p$   and induces a map from  $\Z_p$ to itself. 

The dynamical system of $\chi$ on $\Z_p$ and its induced dynamics on $\Zp/p^n\Zp$ was studied in \cite{FLpre}.
Actually, the convergent series with coefficient in the integral ring $\mathcal{O}_K$ of a finite extension $K$ of $\Qp$ were  studied  as  dynamical systems on 
$\mathcal{O}_K$ in  \cite{FLpre}.    The system $(\Z_p, \theta)$ is a special case when $K=\Q_p$.  

In the following we shall translate results about $\theta$ in the language of results about $\varphi$. For more details on this subject, the reader may consult \cite{FL11,FLpre}.

For each $n\geq 1$, denote by $\varphi_n$ the induced map of $\varphi$ on $p\Zp/p^n\Zp$, i.e.,
$$\varphi(x \!\!\!\!\mod p^n) \equiv  \varphi(x) \ ({\rm mod}   \ p^n),$$
for all $x\in p\Zp$.
In the sprit of Proposition \ref{prop}, the minimality and minimal decomposition of the system $(p\Zp,\varphi)$ can be deduced form its induced dynamics $(p\Zp/p^n\Zp,\varphi_n)$. So we need to study the finite systems $(p\Zp/p^n\Zp,\varphi_n)$.
To this end, the main idea, which comes from \cite{DZunpu,Zievephd}, is to study the behaviour of systems $(p\Zp/p^n\Zp,\varphi_n)$ by induction. We refer the reader to \cite{FL11,FLpre} for the application of this idea to give a minimal decomposition theorem for any convergent power series in $\Zp[[z]]$.

A collection $\sigma=(x_1, \cdots, x_k)$ of $k$ distinct points in $p\Zp/p^n \Zp $ is  called a
{\em cycle} of $\varphi_n$ of length $k$ or a {\em$k$-cycle} at level $n$, if
$$ \varphi_n(x_1)=x_2, \cdots, \varphi_n(x_i)=x_{i+1}, \cdots, \varphi_n(x_k)=x_1.$$
Set
$$ X_\sigma:=\bigsqcup_{i=1}^k  X_i \ \ \mbox{\rm where}\ \
X_i:=\{ x_i +p^nt+p^{n+1}\Z_p; \ t=0,1,\cdots, p-1\} \subset p\Zp/p^{n+1}\Zp.$$
Then
\[
\varphi_{n+1}(X_i) \subset X_{i+1}  \ (1\leq i \leq k-1) \ \ \mbox{\rm
and}\ \  \varphi_{n+1}(X_k) \subset X_1.\]


In the following we are concerned with   the behavior of the finite dynamics
$\varphi_{n+1}$ on the  set $X_\sigma$ and determine all
cycles in $X_\sigma$ of $\varphi_{n+1}$, which will be  called  {\it lifts} of
$\sigma$. Remark that the length of any lift $\tilde{\sigma}$ of
$\sigma$ is a multiple of $k$.

  Let $\mathbb{X}_i=x_i+p^n \Zp$ be the closed disk of radius $p^{-n}$  corresponding  to $x_i\in \sigma$ and
  $$\mathbb{X}_{\sigma}:=\bigsqcup_{i=1}^{k} \mathbb{X}_i$$ be the clopen set corresponding to the cycle $\sigma$.

Let $\psi:=\varphi^{ k}$ be the $k$-th iteration of $\varphi$. Then, any point in $\sigma$
 is fixed by $\psi_n$, the $n$-th induced map of $\psi$. For any point $x\in \mathbb{X}_{\sigma}$,
 we have $(\psi(x)-x)/p^{n}\in \Zp$. Let
\begin{eqnarray}
& &\a_n(x):=\psi'(x)=\prod_{j=0}^{k-1} \varphi'(\varphi^{ j}(x)) \label{an}\\
& &\b_{n}(x):=\frac{\psi(x)-x}{p^n}=\frac{\varphi^{ k}(x)-x}{p^n}.\label{bn}
\end{eqnarray}
One can check that  $\a_n(x) \ ({\rm mod} \ p)$ is always constant on $\mathbb{X}_{\sigma}$. We should remark that under the condition that $\a_n \equiv 1 \ ({\rm mod} \ p)$,  $\b_{n}  \ ({\rm mod} \ p)$ is  also constant on  $\mathbb{X}_{\sigma}$\cite{FL11,FLpre} .
We distinguish the following four behaviours  of $\varphi_{n+1}$ on $X_{\sigma}$:\\
 \indent {\rm (a)} If $\a_n \equiv 1 \ ({\rm mod} \ p)$  and  $\b_{n} \not\equiv 0 \ ({\rm mod} \ p)$, then $\varphi_{n+1}$ restricted to $X_\sigma$ preserves a single  cycle of length $pk$. In this case we say $\sigma$ {\it grows}.\\
 \indent {\rm (b)} If $\a_n \equiv 1 \ ({\rm mod} \ p)$  and $\b_{n} \equiv 0 \ ({\rm mod} \ p)$, then $\varphi_{n+1}$ restricted to $X_\sigma$
 preserves
 $p$ cycles  of length $k$. In this case we say $\sigma$ {\it splits}. \\
 \indent {\rm (c)} If $\a_n \equiv 0 \ ({\rm mod} \ p)$, then $\varphi_{n+1}$ restricted to $X_\sigma$
 preserves
 one cycle of length $k$ and  the remaining points of $X_\sigma$ are mapped into this cycle.
  In this case we say $\sigma$ {\it grows tails}. \\
  \indent {\rm (d)} If $\a_n \not\equiv0, 1 \ ({\rm mod} \ p)$, then $\varphi_{n+1}$ restricted to
$X_{\sigma}$ preserves one cycle of length $k$ and $\frac{p-1}{\ell}$ cycles
of length $k\ell$, where $\ell$ is the order of $\a_n$ in the multiplicative group $\mathbb{F}_p\setminus \{0\}$. In this case we say $\sigma$ {\it partially splits}.

For $n\geq 1$, let $\sigma=(x_1,\dots,x_k)\subset p\Zp/p^n \Zp$ be a $k$-cycle and
let $\tilde{\sigma}$ be a lift of $\sigma$ which is a  of  $\varphi_{n+1}$ contained in
$X_{\sigma}$.
To illustrate the change of nature for a cycle to
its lifts,  we shall show the relationship between $(\a_n, \b_{n})$ and $(\a_{n+1},
\b_{n+1})$.  For the calculation of $(\a_{n+1} \!\!\!\mod p,
\b_{n+1}\!\!\!\mod p)$ from  $(\a_n \!\!\!\mod p, \b_{n} \!\!\!\mod p)$,  the interested reader may consult to \cite[Lemma 2]{FL11} and \cite[p.1636]{FLpre}.
 The following propositions predict  the behaviours of the lifts of a cycle $\sigma$.  
 
\begin{proposition}[\cite{FL11} Proposition 1, and \cite{FLpre} Proposition 4.4]
Let $\sigma $ be a  cycle of $\varphi_n$.\\
 \indent {\rm (1)} 
 If $\sigma$ grows or splits, then any lift $\tilde{\sigma}$ grows
or splits. \\
 \indent {\rm (2)}
If $\sigma$ grows tails, then the single lift $\tilde{\sigma}$
also grows tails. \\
 \indent {\rm (3)}
 If $\sigma$ partially splits, then the
 lift $\tilde{\sigma}$
 of the same length as $\sigma$ partially splits, and the other lifts  of length $k\ell$ grow or split.
\end{proposition}
\begin{proposition} [\cite{FL11, FLpre}] \label{minimalphi}
 Let $\sigma $ be a growing cycle of $\varphi_n$ and $\tilde{\sigma}$ be  the unique lift of $\sigma$. \\
\indent {\rm (1)} If $p >3$ and $n\geq 1$, then $\tilde{\sigma}$ also grows. \\
\indent {\rm (2)} If $p=3$ and $n \geq  2$, then $\tilde{\sigma}$ also grows. \\
\indent {\rm (3)} If $p=2$ and  $\tilde{\sigma}$ grows, then the lift of $\tilde{\sigma}$ grows.
\end{proposition}
 For more details of Proposition \ref{minimalphi}, we refer the reader to \cite[Proposition 2]{FL11}
 for the first two assertion, and to \cite[Corollary 1]{FL11} for the third assertion.
 Remark that the results  in Proposition \ref{minimalphi} are presented only for polynomials in $\Z_p[z]$. Actually, these results also hold for convergent series 
 in $\Z_{p}[[z]]$ (see \cite[Propositions 4.7, 4.8, 4.10, 4.11]{FLpre} for a more general setting).

Let $E\subset p\Zp$ be a $\varphi$-invariant compact open set. Let 
$$E/p^n\Zp:=\{x\in \Zp/p^n\Zp:  \exists \ y \in E \text{ such that }  x\equiv y \ ({\rm mod} \  \ p^n)\}.$$  It is now well known that the
subsystem $(E, \varphi )$ is minimal if and only if the induced map $\varphi_{n} :E/p^n\Zp \rightarrow E/p^n\Zp$ is minimal
 for any $n\geq 1$ (see \cite[p.\,111]{Ana94}, \cite[Theorem 1.2]{Ana06} and \cite[Corollary 4]{CFF09}). Then by Proposition  \ref{minimalphi}, a criterion of the minimality of the system $\varphi$ 
on $p\Zp$ can be obtained.

\begin{corollary}\label{mininduce}
The dynamical system $(p\Zp,\varphi)$ is minimal if and only if\\
\indent {\rm (1)} the finite system $(p\Zp/p^3\Zp,\varphi_3)$  is minimal for $p=2$ or $3$;\\
\indent {\rm (2)} the finite system $(p\Zp/p^2\Zp,\varphi_2)$ is minimal for $p\geq 5$.
\end{corollary}

As we investigate the induced dynamical systems level by level, the possible periods of $\varphi$ on $p\Z_p$ can also be obtained, as a consequence of Proposition \ref{minimalphi}.
\begin{corollary}\label{periodiclength}
Let $\varphi(z)=\sum_{i=0}^{\infty}\lambda_i z^i \in \Zp[[z]]$ with $\lambda_0 \in p\Zp$ and consider the dynamical system $(p\Zp,\varphi)$.\\
\indent { \rm (1)} If $p >3$, the period of a periodic orbit must be a factor of $p-1$.
 \\
\indent { \rm  (2)} If $p=3$, the period of a periodic orbit must be $1, 2$ or $3$. 

\indent { \rm (3)} If $p=2$, the period of a periodic orbit must be $1$ or $2$.
\end{corollary}

Furthermore, the dynamical  structure of $\varphi$ on $p\Z_p$ is fully illustrated by  the following minimal decomposition.
\begin{theorem}[\cite{FLpre}, Theorem 1.1]\label{Zpdeco}
Let $\varphi(z)=\sum_{i=0}^{\infty}\lambda_i z^i \in \Zp[[z]]$ with $\lambda_0 \in p\Zp$. Suppose  $\varphi^n\neq \mbox{id}$ for all $n\geq 1$.
We have the following decomposition
$$p\Zp=\mathcal{P}\bigsqcup \mathcal{M}\bigsqcup \mathcal{B},$$ where $\mathcal{P}$ is the finite set consisting of all periodic points of
$\varphi$, $\mathcal{M}= \bigsqcup_i \mathcal{M}_i$ is the union of all (at most countably
many) clopen invariant sets such that each $\mathcal{M}_i$ is a finite union
of balls and each subsystem $\varphi: \mathcal{M}_i\to \mathcal{M}_i$ is minimal, and 
the points in $ \mathcal{B}$ lie in the attracting basin of $\mathcal{P}\bigsqcup \mathcal{M}$.
\end{theorem}
 The sets $\mathcal{M}_i$
in the above decomposition are called {\em minimal components}.
To completely characterize the dynamical system $(p\Z_p,\varphi)$, each minimal component is described as the adding machine on an odometer  by  giving the structure sequence  of the odometer.
\begin{theorem}[\cite{FLpre}, Theorem 1.1]\label{structure1}
Let $\varphi(z)=\sum_{i=0}^{\infty}\lambda_i z^i \in \Zp[[z]]$ with $\lambda_0 \in p\Zp$.
 If $E\subset p\Z_p$ is a minimal clopen invariant set of $\varphi$,
then $\varphi : E \to E$ is conjugate to the adding machine on the odometer
$\mathbb{Z}_{(p_s)}$, where  $$(p_s) = ( \ell,  \ell p, \ell p^2,
\cdots)$$ where $\ell \geq 1$ is an integer dividing  $p-1$.
\end{theorem}

\medskip

\section{Proof of Theorems \ref{minimal}--\ref{structure-minimal}}\label{pro}
\begin{proof}[Proof of Theorem \ref{minimal}]
We begin with  proving  the ``\,if\," part of the theorem. The reduction $\overline{\phi}$ is minimal  on  $\mathbb{P}^{1}(\mathbb{F}_p)$ which implies that
$\phi_1$  is minimal on $\mathfrak{B}_1$. Let $\psi=\phi^{p+1}$ be the $p+1$-th iteration of $\phi$. Then the ball $B_{1}(0)=p\Zp$ is an invariant set of $\psi$.
 By the argument of Section \ref{induc}, $\psi$ can be written as a  power series with coefficients in $\Zp$ which converges on $p\Zp$,  i.e.
$$\psi(z)=\lambda_0+ \lambda_1 z+ \lambda_2 z^2+ \lambda_3z^3+\cdots $$
with $\lambda_i \in \Zp$ for all $i\geq 1$ and $\lambda_0\in p\Zp$. Noting that $\phi$ is $1$-Lipschitz,  $\psi$ is also $1$-Lipschitz. 
 The minimality of the dynamical  system $(p\Zp, \psi)$ leads to   $\psi$ is isometric on $p\Z_p$, see \cite{Anashin-Khrennikov-AAD,CFF09}. So $\phi$ is isometric on $\P$ with respect to the spherical metric. Hence, $\psi$ is minimal on each ball of radius $1/p$ which 
  implies the minimality of  $(\P,\phi)$. It  suffices to show that
the dynamical  system $(p\Zp, \psi)$ is minimal.

For $n=1$, the cycle $(0)$ is the unique cycle of $\psi_1$. The condition $(\phi^{(p+1)})^{\prime}(0)\equiv1 \ ({\rm mod} \  p)$  implies that the cycle $(0)$ grows or splits.  By the condition $|\phi^{(p+1)}(0)|_p= 1/p$, the cycle $(0)$ grows which then means that the dynamical system $(p\Zp/p^2\Zp,\psi_2)$ is minimal.

For the cases $p>3$,  it follows from Corollary \ref{mininduce} that   the dynamical system $(p\Zp,\psi)$ is minimal.

Note that $(\phi^{p(p+1)})^{\prime}(0)\equiv (\phi^{(p+1)})^{\prime}(0)\equiv 1 \ ({\rm mod} \  \ p)$.
For the cases $p=2 $ or $3$, the additional condition implies that the unique lift of $(0)$ at level $2$ grows.
Thus the system $(p\Zp/p^3\Zp,\psi_n)$ is minimal. By Corollary \ref{mininduce},   the dynamical system $(p\Zp,\psi)$ is minimal.

\smallskip

Now we prove the ``\,only if\," part. Let $\phi$ be a minimal rational map with good reduction. By Proposition \ref{prop}, the system $(\mathfrak{B}_n,\phi_n)$ is minimal for all integers $n\geq 1$. The minimality of system $(\mathfrak{B}_1,\phi_1)$ implies that the reduction $\overline{\phi}$ is minimal on  $\mathbb{P}^{1}(\mathbb{F}_p)$, while the minimality of system $(\mathfrak{B}_2,\phi_2)$ implies that the cycle $(0)$ of $\psi_1=\phi^{p+1}_1$ grows. Hence we obtain  $(\phi^{(p+1)})^{\prime}(0)\equiv 1 \ ({\rm mod} \  \ p)$ and $|\phi^{(p+1)}(0)-0|_p= 1/p$.  Furthermore, by the minimality of system $(\mathfrak{B}_3,\phi_3)$, the unique lift of $(0)$ of $\phi^{(p+1)}$  at level $2$ grows. So we have $|\phi^{ p(p+1)}(0)|_p= 1/p^2$.
\end{proof}
In the above proof of Theorem \ref{minimal}, we have indeed  established the following result.

\begin{theorem}\label{minifinite}
 Let $\phi:\P\mapsto \P$ be a rational map of $\deg \phi \geq 2$ with good reduction. 
Then the dynamical system $(\P,\phi)$ is minimal if and only if  the following condition is satisfied:
\begin{itemize}
  \item[(1)]  the system $(\mathfrak{B}_3, \phi_3)$ is minimal for $p=2$ or $3$,
  \item[(2)] the system $(\mathfrak{B}_2, \phi_2)$ is minimal for $p\geq 5$.
\end{itemize}
\end{theorem}
\begin{proof}[Proof of Theorem \ref{minmaldecomp}]
The space $\P$ is a union of $p+1$
balls with radius $p^{-1}$. Each ball  can be identified with a point in
$\PF$. The reduction map $\overline{\phi}$ on $\PF$ admits some
cycles. By iteration of $\overline{\phi}$, the points outside the cycles are attracted by the cycles. The
ball corresponding to such a point  will be put into the third part
$\mathcal{B}$ in the minimal  decomposition.

Let $(\overline{x}_1,\cdots,\overline{x}_k)\subset \PF$ be a cycle of $\overline{\phi}$. Without loss of generality,
we assume  $\overline{x}_1=0$. Let $\psi=\phi^{ k}$ be the $k$-th iteration of $\phi$. Then  the ball $B_{1}(0)=p\Zp$ is an invariant set of $\psi$. Noting  that $\phi$ is a rational map of $\deg \phi \geq 2$, $\phi^n$ is a rational map of degree $(\deg \phi)^n$. So $\phi^n\neq id$ for all $n\geq 1$.
By Theorem \ref{Zpdeco}, the dynamical system $(p\Zp,\psi)$ has a minimal decomposition which then gives a  minimal decomposition of the system $(\bigsqcup_{i=1}^{k} B_1(x_i),\phi)$.

Applying  the same argument  to all the cycles of $\overline{\phi}$,  we obtain  the minimal decomposition of the whole system $(\P,\phi)$.

Moreover, by Corollary \ref{periodiclength}, the possible lengths of periodic orbits are obtained.  
\end{proof}

\begin{proof}[Proof of Theorem \ref{structure-minimal}]
Let $k$  denote the length  of the induced periodic orbit  of $\phi$ on  the minimal set $E$  at  the first level. Each  point at this first level
corresponds to a point in $\mathbb{P}^{1}(\mathbb{F}_p)$.  So $k\leq p+1$.

Without loss of generality, we assume that $0\in E$. The set $B_1(0)\cap E$ is invariant by $\phi^k$. Consider the dynamical system $(B_1(0)\cap E, \phi^k )$.
Noting $\phi$ have good reduction, by (\ref{inte-coe})  and Theorem \ref{structure1}, we deduce  that 
 $\phi^k : B_1(0)\cap E \to B_1(0)\cap E$ is conjugate to the adding machine on the odometer
$\mathbb{Z}_{(p_s)}$, where  $$(p_s) = (\ell, \ell p, \ell p^2,\cdots)$$ with $\ell \mid (p-1)$,
which implies $\phi : E \to E$ is conjugate to the adding machine on the odometer
$\mathbb{Z}_{(p_s)}$, where  $$(p_s) = (k, k\ell, k \ell p, k \ell p^2,\cdots).$$
\end{proof}

\medskip

\section{Minimal rational map for the case $p=2$}\label{pequal2}
\subsection{Minimal Conditions}

Let $$\phi(z)=\frac{a_0+a_1z+\cdots +a_{d-1}z^{d-1}+a_d z^d}{b_1z+\cdots + b_{d-1}z^{d-1}+ b_d z^d}$$
be a rational map of  degree $d\geq 2$ with $a_i,b_j \in \Q_2$  and $a_d=b_d=1$. Recall that 
 $A_\phi=\sum_{i\geq 0}a_i$,
$B_\phi=\sum_{j\geq 1}b_j$,  $A_{\phi,1}=\sum_{i\geq 0}a_{2i+1}$, $A_{\phi,2}=\sum_{i\geq 0}a_{4i+1}$ and $A_{\phi,3}=\sum_{i\geq 0}a_{4i+3}$.

 Since $\phi(0)=\infty$, for the convenience of calculation,  we shall select a suitable coordinate.  Let 
\begin{equation}\label{changecor1}
\psi(z):=\frac{1}{\phi(z)}=\frac{b_1z+\cdots + b_{d-1}z^{d-1}+ b_d z^d}{a_0+a_1z+\cdots +a_{d-1}z^{d-1}+a_d z^d}
\end{equation}
and 
\begin{equation}\label{changecor2}
\varphi(z):=\phi(\frac{1}{z})=\frac{a_d +a_{d-1}z +\cdots +a_1z^{d-1}+ a_0z^d}{ b_d+b_{d-1}z+ \cdots +  b_1z^{d-1}}.
\end{equation} 
Then we have $\phi^{3}=\phi\circ \varphi\circ \psi$.
\begin{lemma}\label{mod1}
If  $\phi$ has good reduction  and  $(\mathfrak{B}_1,\phi_1)$ is minimal, then 
\begin{equation}\label{conditionmod1}
\left\{
   \begin{array}{ll}
     a_i,b_j \in \Z_2, \hbox{~~~~~~~for $0\leq i \leq d-1$ and $1\leq j \leq d-1$},\\
     a_{0}\equiv 1 \ ({\rm mod}   \ 2),  \\
     A_\phi \equiv 0 \ ({\rm mod}   \ 2),\\
     B_\phi \equiv 1 \ ({\rm mod}   \ 2). \\
   \end{array}
 \right.
\end{equation}
Conversely, the condition {\rm (\ref{conditionmod1})} implies that   $\phi$ is $1$-Lipschitz continuous  and  $(\mathfrak{B}_1,\phi_1)$ is minimal. Moreover, the Taylor expansion of $\phi^3$  at $0$ is of the form
$$\phi^3(z)=\phi(1)+\lambda z+ \lambda_2 z^2+ \lambda_3z^3+\cdots,
$$
where $\lambda,\lambda_2,\lambda_3\cdots \in \Z_2 $. 
\end{lemma}
\begin{proof}
Assume that  $\phi$ has good reduction,  the coefficients $a_i$ and $b_j$ are in $\Z_2$. Otherwise, the degree of the reduction map of $\phi$ is strictly less than $d$, which implies that  $\phi$ has bad reduction.

Let $$\overline{\phi}(z)=\frac{\overline{f}(z)}{\overline{g}(z)}=\frac{\overline{a}_0+\overline{a}_1z+\cdots +\overline{a}_{d-1}z^{d-1}+z^d}{\overline{b}_1z+\cdots + \overline{b}_{d-1}z^{d-1}+ z^d}$$
be the reduction of $\phi$ modulo $p$. The condition that $\phi$ has good reduction implies that  the polynomials $\overline{f},\overline{g}$ have no common zero. 
As we have already indicated above, we can assume $\phi(0)=\infty$ and $\phi(\infty)=1$.
So we have $a_{0}\equiv 1 \ ({\rm mod}  \ 2)$.
The minimality of the system $(\mathfrak{B}_1,\phi_1)$ means that  $\overline{\phi}(1)=0$. So we have $\overline{f}(1)=0$ which implies $ A_\phi \equiv 0\ ({\rm mod} \  \ 2) $.
Since the polynomials $\overline{f},\overline{g}$ have no common zero, we have $\overline{g}(1)\neq 0$ which means that  $B_{\phi} \equiv 1 \ ({\rm mod} \   2)$.

Now we assume that  the coefficients of  $\phi$  satisfy the condition (\ref{conditionmod1}). 
By (\ref{conditionmod1}), we can check directly that $(\mathfrak{B}_1,\phi_1)$ is minimal. 

By Lemma \ref{lem-1-Lip}, the conditions 
$A_\phi \equiv 0 \ ({\rm mod}   \ 2),$ and $B_\phi \equiv 1 \ ({\rm mod}   \ 2)$ imply that $\phi(z+1)$ is $1$-Lipschitz from $p\Zp$ to $p\Zp$ and can be written as the form of (\ref{inte-coe}).
Recall that $\phi^3=\phi \circ \varphi \circ \psi$ and note that $\psi(0)=0$ and $\varphi(0)=1$. It suffices to prove that $\psi$ and $\varphi-1$ are both $1$-Lipschitz from $p\Zp$ to itself and have the form (\ref{inte-coe}). These can be obtained by direct calculations and by applying Lemma \ref{lem-1-Lip}.
\end{proof}

In the following lemma, we will calculate the derivative of $\phi^{3}$ at $0$.  Using the chain rule, we obtain 
\begin{equation}
(\phi^{3})^{\prime}(0)=\psi^\prime(0)\cdot\varphi^\prime(0)\cdot\phi^\prime(1).
\end{equation}\label{derivative}

For simplicity, we denote $A^{\prime}_{\phi}:=\sum_{i\geq1}ia_i$ and
$B^{\prime}_{\phi}:=\sum_{i\geq1}ib_i$.
By  calculation, we have 
$$\eta_1:=\psi^{\prime}(0)=b_1/a_0,$$
$$\eta_2:=\varphi^{\prime}(0)=a_{d-1}-b_{d-1},$$
and $$\eta:=\phi^{\prime}(1)=\frac{A_\phi^\prime B_\phi-B_\phi^\prime A_{\phi}}{B_{\phi}^{ 2}}.$$
\begin{lemma}\label{mod4}
Assume that the  rational map $\phi$ has good reduction. 
 Then the system  $(\mathfrak{B}_2,\phi_2
)$ is minimal if and only if
$$\left\{
   \begin{array}{ll}
     a_i,b_j \in \Z_2, \hbox{~~~~~~~for $0\leq i \leq d-1$ and $1\leq j \leq d-1$},\\
     a_{0}\equiv 1 \ ({\rm mod}   \ 2),  \\
    B_{\phi} \equiv 1 \ ({\rm mod}   \ 2), \\
    A_{\phi} \equiv 2 \ ({\rm mod}   \ 4) \\
    A_{\phi,1} \equiv 1 \ ({\rm mod}   \ 2),\\
     b_1 \equiv 1 \ ({\rm mod}   \ 2),\\
     a_{d-1}-b_{d-1} \equiv 1 \ ({\rm mod}   \ 2).\\
   \end{array}
 \right.$$
\end{lemma}
\begin{proof}The rational map  $\phi$ has good reduction and $(\mathfrak{B}_2, \phi_2)$ is minimal implies that  $(\mathfrak{B}_1, \phi_1)$ is also minimal.
So $p\Zp$ is $\phi^{3}$-invariant. By the classification of the lifts of the cycles,  $(\mathfrak{B}_2,\phi_2
)$ is minimal if and only if
\begin{align} \label{mod41}
  (\mathfrak{B}_1, \phi_1) \mbox{~is  minimal}, \ (\phi^{3})^{\prime}(0)\equiv 1\ ({\rm mod}   \ 2), \ \mbox{~and~} \phi^3(0) \phi(1) \equiv 2\ ({\rm mod}   \ 4)
  \end{align}
under the condition that $\phi$ has good reduction.

 By Lemma \ref{mod1}, $(\mathfrak{B}_1, \phi_1)$ is  minimal if and only if   $a_0\equiv 1 \ ({\rm mod}   \ 2)$, $B_{\phi}\equiv 1\ ({\rm mod}   \ 2)$  and $A_{\phi} \equiv 0 \ ({\rm mod}   \ 2)$.

Recall that
\begin{align*}
 (\phi^{ 3})^{\prime}(0)&= \psi^\prime(0)\cdot\varphi^\prime(0)\cdot\phi^\prime(1)
 = \frac{b_1}{a_0}(a_{d-1}-b_{d-1})\frac{A_\phi^\prime B_\phi-B_\phi^\prime A_{\phi}}{B_{\phi}^{ 2}}.
 \end{align*}
  The second condition $(\phi^{3})^{\prime}(0)\equiv 1\ ({\rm mod}   \ 2)$ of (\ref{mod41}) 
 implies that
 $$b_1\equiv 1 \ ({\rm mod}   \ 2),$$
$$(a_{d-1}-b_{d-1})\equiv 1 \ ({\rm mod}   \ 2),$$ and
$$A^{\prime}_{\phi} \equiv 1 \ ({\rm mod}   \ 2).$$
Since
\begin{align*}
A^{\prime}_\phi &= \sum_{i\geq 0}(2i+1)a_{2i+1}+\sum_{i\geq1}2ia_{2i}
\equiv   \sum_{i\geq 0} a_{2i+1}\ ({\rm mod}   \ 2),
\end{align*}
we have  $A_{\phi,1}  \equiv 1 \ ({\rm mod}   \ 2)$.

By the last condition  $\phi(1)\equiv 2 \ ({\rm mod}   \ 4)$ in  (\ref{mod41}), we immediately get $$A_{\phi} \equiv 2 \ ({\rm mod}   \ 4).$$
\end{proof}
For  the proof of Theorem \ref{p=2}, we need calculate the second derivative of $\phi^{3}$ at $0$. For simplicity, denote $A^{\prime\prime}_{\phi}:=\sum_{i\geq 2}i(i-1)a_i$ and
$B^{\prime\prime}_{\phi}:=\sum_{i\geq2}i(i-1)b_i$.
Before the proof let us first calculate the second derivatives $\psi^{\prime\prime}(0)$, $\varphi^{\prime\prime}(0)$ and $\phi^{\prime\prime}(1)$:
$$\xi_1:=\psi^{\prime\prime}(0)=\frac{2b_2a_0^2-2a_1b_1a_0}{a_0^3},$$
$$\xi_2:=\varphi^{\prime\prime}(0)=2(a_{d-2}-b_{d-2})+2(b_{d-1}^2-a_{d-1}b_{d-1}),$$
$$\xi:=\phi^{\prime\prime}(1)=\frac{A^{\prime\prime}_{\phi}B_{\phi}^{2}-B^{\prime\prime}_{\phi}  A_\phi B_\phi + 2(A_\phi (B_\phi^{\prime})^2-A_\phi^{\prime}B_\phi^{\prime}B_\phi)}{B_{\phi}^{3}}.$$
Now we are ready to prove  Theorem \ref{p=2}.
\begin{proof}[Proof of Theorem \ref{p=2}]
Under the condition that $\phi$ has good reduction, Theorem \ref{minifinite} implies that $(\P,\phi)$ is minimal if and only if
$(\mathfrak{B}_3,\phi_3)$ is minimal which is equivalent to the following three conditions 
\begin{align}\label{level3-cond}
(\mathfrak{B}_2,\phi_2) \mbox{ is minimal}, \ (\phi^{ 6})^\prime(0)\equiv 1 \ ({\rm mod} \   2), \ \mbox{ and }\ \phi^{ 6}(0) \equiv 4 \ ({\rm mod} \   8).
\end{align}

In the following, we will characterize the three conditions of (\ref{level3-cond}).

Firstly, the minimality of the system $(\mathfrak{B}_2,\phi_2) $ has already  been  characterized   by coefficients in Lemma \ref{mod4}. Then we obtain all the conditions in (\ref{equ1}) except the last one. 

Secondly, if the first condition of (\ref{level3-cond}) is satisfied then by (\ref{mod41}), $$(\phi^{ 6})^\prime(0)\equiv ((\phi^{3})^\prime(0))^2\equiv 1 \ ({\rm mod} \   2).$$ Thus the second condition of (\ref{level3-cond}) is in fact included in the first condition of (\ref{level3-cond}).

 Finally, we are going to investigate  the relation among the coefficients by using the last condition $\phi^{ 6}(0) \equiv 4 \ ({\rm mod} \   8)$ of (\ref{level3-cond}) which will lead to the last condition of (\ref{equ1}).
 
 Note that all  the Taylor's coefficients of $\phi^3$ expanded at $z=0$ belong to $\Z_2$. Hence   for $z\in 2\Z_2$, we have 
$$\phi^{3}(z)\equiv \phi(1)+(\phi^{3})^{\prime}(0) z+\frac{(\phi^{3})^{\prime\prime}(0)}{2}z^2 \ ({\rm mod}   \ 8).$$
 Thus
\begin{align}\label{dividing}
\phi^{ 6}(0)\equiv \phi(1)+(\phi^{3})^{\prime}(0) \phi(1)+\frac{(\phi^{3})^{\prime\prime}(0)}{2}\phi(1)^2 \ ({\rm mod}   \ 8).
\end{align}
Notice that  $\phi(1) \equiv 2 \ ({\rm mod} \   4) $ and $(\phi^{3})^{\prime}(0)\equiv 1 \ ({\rm mod} \  2)$.
Then dividing both sides of (\ref{dividing}) by $\phi(1)$, we deduce that the condition $\phi^{ 6}(0) \equiv 4 \ ({\rm mod} \  \ 8)$ is  equivalent to 
$$1+(\phi^{3})^{\prime}(0) +\frac{(\phi^{3})^{\prime\prime}(0)}{2}\phi(1)\equiv 2 \ ({\rm mod}   \ 4),$$
which  in turn, by noting that $(\phi^{3})^{\prime \prime}(0)/2 \in \Z_2$,
is  equivalent to
\begin{align}\label{equiv-conditions}
(\phi^{3})^{\prime}(0)+(\phi^{3})^{\prime\prime}(0)\equiv 1 \ ({\rm mod}   \ 4).
\end{align}

To continue, let us calculate $(\phi^{3})^{\prime}(0)\ ({\rm mod} \   4), (\phi^{3})^{\prime\prime}(0) \ ({\rm mod} \   4)$  and simplify their expressions.
Using $a_0\equiv 1 \ ({\rm mod} \   2)$ and $B_\phi\equiv 1\ ({\rm mod} \   2)$, we have $1/a_0\equiv a_0\ ({\rm mod} \   4)$ and $B_{\phi}^{2}\equiv 1 \ ({\rm mod} \   4)$. Then $$ (\phi^{3})^{\prime}(0)\equiv  a_0b_1(a_{d-1}-b_{d-1})(A_{\phi}^{\prime}B_\phi-A_\phi B_{\phi}^{\prime}) \ ({\rm mod}   \ 4).$$
Since $A_\phi\equiv2 \ ({\rm mod} \   4)$ and $ \alpha_1\equiv \alpha_2\equiv 1 \ ({\rm mod} \  2)$,  we have
\begin{align*}
(\phi^{3})^{\prime}(0)&\equiv  a_0b_1(b_{d-1}-b_{d-1})A_{\phi}^{\prime}B_\phi-A_\phi B_{\phi}^{\prime} \ ({\rm mod}   \ 4).
\end{align*}

For the  second derivative of $\phi^{3}$ at $0$, by the chain rule,  we have 
\begin{align}\label{5.3}
(\phi^{3})^{\prime\prime}(0)
= \eta\eta_2\xi_1+\eta\eta_1^2\xi_2+\xi\eta_1^2\eta_2^2. 
\end{align}
Since $\eta\equiv \eta_1\equiv \eta_2\equiv 1 \ ({\rm mod} \   2)$, we have $\eta^2\equiv \eta_1^2\equiv \eta_2^2\equiv 1 \ ({\rm mod} \   4)$.
Consequently,  we obtain
\begin{align}\label{5.4}
(\phi^{3})^{\prime\prime}(0)&\equiv  \eta\eta_2\xi_1+\eta\xi_2+\xi \ ({\rm mod}   \ 4).
\end{align}
Apply  the expressions of $\xi,\xi_1$ and $\xi_2$ to Equation (\ref{5.4}). Using the facts  that  $a_0\equiv b_1 \equiv a_{d-1}-b_{d-1}\equiv 1 \ ({\rm mod} \   2)$ and $B_{\phi}\equiv 1 \ ({\rm mod} \   2)$,  we deduce
\begin{align*}
(\phi^{3})^{\prime\prime}(0)&\equiv  2\eta\eta_2(a_0b_2-a_1b_1)+2\eta(a_{d-2}-b_{d-2}+b_{d-1}^{2}-a_{d-1}b_{d-1})+\\
&\ \ \ \ \ \ \ \ \ \ \ \ \ \ \ +A_{\phi}^{\prime\prime}B_{\phi}-2A_{\phi}^{\prime}B_{\phi}^{\prime}-A_{\phi}B_{\phi}^{\prime\prime} \ ({\rm mod}   \ 4)\\
&\equiv  2(b_2-a_1+a_{d-2}-b_{d-2}+b_{d-1})+A_{\phi}^{\prime\prime}B_{\phi}-2A_{\phi}^{\prime}B_{\phi}^{\prime} -A_{\phi}B_{\phi}^{\prime\prime} \ ({\rm mod}   \ 4).
\end{align*}
Observing  that   $B_{\phi}^{\prime\prime}\equiv 0\ ({\rm mod} \   2)$ and $A_{\phi}\equiv 0 \ ({\rm mod} \   2)$,  we have 
\begin{align*}
 &  (\phi^{3})^{\prime}(0)+(\phi^{3})^{\prime\prime}(0)\\
\equiv &\    a_0b_1(a_{d-1}-b_{d-1})A_{\phi}^{\prime}B_\phi-A_\phi B_{\phi}^{\prime}+2(b_2-a_1+a_{d-2}-b_{d-2}+b_{d-1})\\
&  \ \ \ \ \ \ \ \ \ \ \ \ \ \ \ +A_{\phi}^{\prime\prime}B_{\phi}-2A_{\phi}^{\prime}B_{\phi}^{\prime} \ ({\rm mod}   \ 4).
\end{align*}
Since $A_{\phi}\equiv 2 \ ({\rm mod} \   4) $ and  $A_{\phi}^{\prime}\equiv  A_{\phi,1} \equiv 1 \ ({\rm mod}  \ 2)$, we get
\begin{align*}
 &  (\phi^{3})^{\prime}(0)+(\phi^{3})^{\prime\prime}(0)\\
\equiv &\  a_0b_1(a_{d-1}-b_{d-1})A_{\phi}^{\prime}B_\phi+2(b_2-a_1+a_{d-2}-b_{d-2}+b_{d-1})+A_{\phi}^{\prime\prime}B_{\phi} \ ({\rm mod}   \ 4).
\end{align*}
Note that \begin{align*}
A_{\phi}^{\prime} \equiv \sum_{i\geq 0} a_{4i+1}+ 2\sum_{i\geq 0}a_{4i+2}-\sum_{i\geq 0}a_{4i+3} \ ({\rm mod}   \ 4)
\end{align*}
and
\begin{align*}
A_{\phi}^{\prime\prime}\equiv &2\sum_{i\geq 0}a_{4i+2}+2\sum_{i\geq 0}a_{4i+3} \ ({\rm mod}   \ 4).
\end{align*}
Then
\begin{align*}
 &  (\phi^{3})^{\prime}(0)+(\phi^{3})^{\prime\prime}(0)\\
\equiv&  a_0b_1(a_{d-1}-b_{d-1})(A_{\phi,2}-A_{\phi,3})B_\phi+2(b_2-a_1+a_{d-2}-b_{d-2}+b_{d-1})\\
&\ \ \ \ \ \ \ \ \ \ \ \ \ \ \  +2\sum_{i\geq 0}a_{4i+2}+2(\sum_{i\geq 0}a_{4i+2}+A_{\phi,3}) \ ({\rm mod}   \ 4)\\
\equiv& a_0b_1(a_{d-1}-b_{d-1})(A_{\phi,2}-A_{\phi,3})B_\phi \\
&\ \ \ \ \ \ \ \ \ \ \ \ \ \ \  +2(b_2-a_1+a_{d-2}-b_{d-2}+b_{d-1}+A_{\phi,3})  \ ({\rm mod}   \ 4).
\end{align*}
By (\ref{equiv-conditions}), we obtain the last condition of (\ref{equ1}). 
Hence, we obtain the first part of Theorem \ref{p=2}.

For the other part of Theorem \ref{p=2}, note that the first four conditions of (\ref{equ1}) imply the $1$-Lipschitz continuity of $\phi$. Further, $\phi^3$ can be written as a form of power series with coefficients in $\Zp$. Thus from the proof of Theorem \ref{minimal}, we deduce that the minimality of $\phi$ is equivalent to the conditions (\ref{level3-cond}). Hence the above proof of the first part leads to the conclusion of the second part of the theorem.
\end{proof}

\begin{proof}[Proof of Corollary \ref{non-exist}]
Remark that the good reduction property and minimality of $\phi$ imply that $\phi$ is isometric on $\mathbb{P}^{1}(\Q_2)$.    

For a  rational map $\phi$ of degree 2 which has  good reduction   and is  minimal on $\mathbb{P}^1(\Q_2)$, each  point   has either  $0$ or $2$ pre-images since $\phi$ can not have critical point  in $\mathbb{P}^{1}(\Q_2)$. This contradicts  to the fact 
that $\phi$ is isometric. 

For a rational map $\phi$ of degree $3$  which is defined as (\ref{phi-def}), the good reduction property and minimality give  
the reduction $\overline{\phi}(z)=\frac{(z-1)f(z)}{z g(z)}$, where $f,g\in \mathbb{F}_{2}[z]$ are two quadratic irreducible polynomials.  However,  $1+z+z^2$ is the unique quadratic irreducible  polynomial over $\mathbb{F}_2$.  Hence, $f=g$, which contradicts that 
$\phi$ has good reduction.

For a rational map $\phi$ of degree $4$  which is defined as (\ref{phi-def}),  the good reduction property and minimality give  
the reduction $\overline{\phi}(z)=\frac{(z-1)f(z)}{z g(z)}$, with $f,g\in \mathbb{F}_{2}[z]$ being two different cubic  irreducible polynomials.  It is known that $1+z+z^3$ and $1+z^2+z^3$ are the only two cubic  irreducible polynomials  over $\mathbb{F}_{2}$.
Hence, $\overline{\phi}(z)$ can only have two cases: 
\[
\overline{\phi}(z)=\frac{(z-1)(1+z+z^3)}{z (1+z^2+z^3)} \quad \text{or} \quad \overline{\phi}(z)=\frac{(z-1)(1+z^2+z^3)}{z (1+z+z^3)}.
\]
 By simple calculations, in both cases, $a_3-b_3\equiv 0  \ ({\rm mod}   \ 2)$, which contradicts the seventh equation in (\ref{equ1}).
 \end{proof}

\subsection{Minimal rational maps of degree $4$}\label{deg4}

This  section is devoted to investigate the minimal rational maps of degree $4$ which are $1$-Lipschitz continuous and their induced orbits on $\mathfrak{B}_3$.

 For a rational map $$\phi(z)=\frac{a_0+a_1z+\cdots +a_{d-1}z^{d-1}+a_d z^{d}}{b_1z+\cdots + b_{d-1}z^{d-1}+ b_{d} z^{d}}$$ of degree $d$ with $a_i,b_j \in \Q_2$ and $a_{d}=b_{d}=1$. Assume that $\phi$   $1$-Lipschitz  continuous and minimal on $\mathbb{P}^{1}(\Q_2)$. The orbit of $\phi_1$ on $\mathfrak{B}_1$ is $0\to \infty \to 1 \to 0$.
For the convenience of investigating the induced orbit on $\mathfrak{B}_n$ with $n\geq 2$, we choose the local coordinate  around $\infty$ by $f(z)=1/z$.  Denote by $\widetilde{B}_{1}(0)$ the image of $B_{1}(\infty)$ under $f$. we have the following communicating graph. 
 \begin{center}
 \setlength{\unitlength}{1mm}
 \begin{picture}(60,40)

 \put(7,28){$\vector(1,-1){20}$}

 \put(28,4){$\widetilde{B}_{1}(0)$}
 \put(2,30){$B_1(0)$}
 \put(15,14){$\psi$}
\put(31,33){$\phi$}
 \put(10,28){$\vector(3,-1){17}$}
 \put(17,27){$\phi$}
 \put(28,20){$B_1(\infty)$}
 \put(32,19){$\vector(0,-1){10}$}
\put(38,23){$\vector(3,1){17}$}
\put(45,27){$\phi$}
  \put(33,13){$f$}
\put(55,30){$B_{1}(1)$}
\put(37,8){$\vector(1,1){20}$}
 \put(45,14){$\varphi$}
\put(54,31){$\vector(-1,0){43}$}
 \end{picture}
 \end{center}
Actually, the two balls  $B_{1}(0)$ and $\widetilde{B}_{1}(0)$ have no difference. To avoid confusion, the elements in $\widetilde{B}_{1}(0)$ are denoted by $\widetilde{z}$,  such as  $\widetilde{1}$, $\widetilde{2}$.

Let $\widetilde{\mathbb{P}^1}(\Q_2)$ be the disjoint union of $B_{1}(0), \widetilde{B}_{1}(0)$ and $B_{1}(1)$, i.e.
$$\widetilde{\mathbb{P}^1}(\Q_2)=B_{1}(0)\sqcup \widetilde{B}_{1}(0) \sqcup B_{1}(1).$$
Then we can define a map $\widetilde{\phi}$ from  $\widetilde{\mathbb{P}^1}(\Q_2)$ to itself as follows:
$$\widetilde{\phi}(z):=\left\{
                    \begin{array}{ll}
                      \widetilde{\psi(z)}, & \mbox{if $z\in B_1(0)$;} \\
                      \varphi(z), & \mbox{if $z\in \widetilde{B}_{1}(0)$;}\\
                      \phi(z),  & \mbox{if $z\in B_{1}(1)$.}\\
                    \end{array}
                  \right.
 $$

Instead of investigating the dynamical system $(\mathbb{P}^{1}(\Q_2),\phi)$, we study the dynamical system $(\widetilde{\mathbb{P}}^{1}(\Q_2),\widetilde{\phi})$, so that we can do modulo calculation by computer. We  obtain    all     rational  maps of degree 4  with coefficients  in $\{0,1,2,3\}$ which  have good reduction and are minimal on $\mathbb{P}^{1}(\Q_p)$. 
We associate each of the $12$ vertices at level $3$ a label
$$\begin{cases}
    i,  \hbox{\ \ \ \ \ \ where $i \in \{0, 1, 2\cdots 7\}$,  if the vertex corresponds to the ball $B_{n}(i)$},\\
    \widetilde{i}, \hbox{\ \ \ \ \ \ where  $i \in \{0,  2, 4, 6\}$,  if the vertex corresponds to the ball $B_{n}(1/i)$}.\\
   \end{cases}$$
 The rational maps and their induced orbits at level $3$ are showed in the following table. 
\\

\begin{tabular}{|c|l|}
  \hline
   Coefficients $a_0, a_1, a_2, a_3, b_1, b_2, b_3$ & Induced  periodic orbits at  level 3\\
\hline
$1,\ \ 0,\ \ 1,\ \ 3,\ \ 3,\ \ 1,\ \ 0$ & \\
 $1,\ \ 1,\ \ 1,\ \ 2,\ \ 3,\ \ 2,\ \ 3$ &\\
 $1,\ \ 2,\ \ 1,\ \ 1,\ \ 3,\ \ 3,\ \ 2$&\\
$1,\ \ 2,\ \ 3,\ \ 3,\ \ 3,\ \ 3,\ \ 0$ & $0\to \widetilde{0}\to 1\to 6 \to \widetilde{6}\to 3 $\\
$1,\ \ 3,\ \ 1,\ \ 0,\ \ 3,\ \ 0,\ \ 1$ & $\to 4\to \widetilde{4}\to 5\to 2\to \widetilde{2}\to 7$   \\
$ 1,\ \ 3,\ \ 3,\ \ 2,\ \ 3,\ \ 0,\ \ 3$ &\\
$3,\ \ 2,\ \ 1,\ \ 3,\ \ 1,\ \ 1,\ \ 0$ &\\
$3,\ \ 3,\ \ 1,\ \ 2,\ \ 1,\ \ 2,\ \ 3$ &\\
\hline
$1,\ \ 0,\ \ 1,\ \ 3,\ \ 1,\ \ 1,\ \ 2$ & \\
 $1,\ \ 1,\ \ 1,\ \ 2,\ \ 1,\ \ 2,\ \ 1$ &\\
 $1,\ \ 2,\ \ 1,\ \ 1,\ \ 1,\ \ 3,\ \ 0$&\\
$1,\ \ 2,\ \ 3,\ \ 3,\ \ 1,\ \ 3,\ \ 2$ & $0\to\widetilde{0}\to1\to6\to\widetilde{2}\to3$\\
$1,\ \ 3,\ \ 1,\ \ 0,\ \ 1,\ \ 0,\ \ 3$ &$\to4\to\widetilde{4}\to5\to2\to\widetilde{6}\to7$\\
$1,\ \ 3,\ \ 3,\ \ 2,\ \ 1,\ \ 0,\ \ 1$ &\\
$3,\ \ 2,\ \ 1,\ \ 3,\ \ 3,\ \ 1,\ \ 2$&\\
$3,\ \ 3,\ \ 1,\ \ 2,\ \ 3,\ \ 2,\ \ 1$ &\\
\hline
$1,\ \ 0,\ \ 1,\ \ 3,\ \ 1,\ \ 1,\ \ 0$ & \\
$1,\ \ 1,\ \ 1,\ \ 2,\ \ 1,\ \ 2,\ \ 3$ &\\
$1,\ \ 2,\ \ 1,\ \ 1,\ \ 1,\ \ 3,\ \ 2$&\\
$1,\ \ 2,\ \ 3,\ \ 3,\ \ 1,\ \ 3,\ \ 0$ & $0\to \widetilde{0}\to1\to2\to\widetilde{6}\to3$\\
$1,\ \ 3,\ \ 1,\ \ 0,\ \ 1,\ \ 0,\ \ 1$ & $\to4\to\widetilde{4}\to5\to6\to\widetilde{2}\to7$\\
$1,\ \ 3,\ \ 3,\ \ 2,\ \ 1,\ \ 0,\ \ 3$ &\\
$3,\ \ 2,\ \ 1,\ \ 3,\ \ 3,\ \ 1,\ \ 0$&\\
$3,\ \ 3,\ \ 1,\ \ 2,\ \ 3,\ \ 2,\ \ 3$ &\\
\hline

\end{tabular}

\begin{tabular}{|c|l|}
  \hline
   Coefficients $a_0, a_1, a_2, a_3, b_1, b_2, b_3$ & Induced  periodic orbits at  level 3\\

\hline
$1,\ \ 0,\ \ 1,\ \ 3,\ \ 3,\ \ 1,\ \ 2$ & \\
$1,\ \ 1,\ \ 1,\ \ 2,\ \ 3,\ \ 2,\ \ 1$ &\\
$1,\ \ 2,\ \ 1,\ \ 1,\ \ 3,\ \ 3,\ \ 0$ &\\
$1,\ \ 2,\ \ 3,\ \ 3,\ \ 3,\ \ 3,\ \ 2$ & $0\to\widetilde{0}\to1\to2\to\widetilde{2}\to3$\\
$1,\ \ 3,\ \ 1,\ \ 0,\ \ 3,\ \ 0,\ \ 3$ &$\to4\to\widetilde{4}\to5\to6\to\widetilde{6}\to7$\\
$1,\ \ 3,\ \ 3,\ \ 2,\ \ 3,\ \ 0,\ \ 1$ &\\
$3,\ \ 2,\ \ 1,\ \ 3,\ \ 1,\ \ 1,\ \ 2$&\\
$3,\ \ 3,\ \ 1,\ \ 2,\ \ 1,\ \ 2,\ \ 1$ &\\
\hline
$1,\ \ 0,\ \ 3,\ \ 1,\ \ 3,\ \ 1,\ \ 0$ & \\
$1,\ \ 1,\ \ 3,\ \ 0,\ \ 3,\ \ 2,\ \ 3$ &\\
$3,\ \ 0,\ \ 1,\ \ 1,\ \ 1,\ \ 3,\ \ 0$ &\\
$3,\ \ 0,\ \ 3,\ \ 3,\ \ 1,\ \ 3,\ \ 2$ & $0\to \widetilde{0}\to1\to 6\to \widetilde{6}\to 7$\\
$3,\ \ 1,\ \ 1,\ \ 0,\ \ 1,\ \ 0,\ \ 3$ &$\to 4\to \widetilde{4}\to 5\to 2\to \widetilde{2}\to 3$\\
$3,\ \ 1,\ \ 3,\ \ 2,\ \ 1,\ \ 0,\ \ 1$ &\\
$3,\ \ 2,\ \ 3,\ \ 1,\ \ 1,\ \ 1,\ \ 0$&\\
$3,\ \ 3,\ \ 3,\ \ 0,\ \ 1,\ \ 2,\ \ 3$ &\\
\hline

$1,\ \ 0,\ \ 3,\ \ 1,\ \ 1,\ \ 1,\ \ 2$ & \\
$1,\ \ 1,\ \ 3,\ \ 0,\ \ 1,\ \ 2,\ \ 1$ &\\
$3,\ \ 0,\ \ 1,\ \ 1,\ \ 3,\ \ 3,\ \ 2$ &\\
$3,\ \ 0,\ \ 3,\ \ 3,\ \ 3,\ \ 3,\ \ 0$ & $0\to\widetilde{0}\to1\to6\to\widetilde{2}\to7$\\
$3,\ \ 1,\ \ 1,\ \ 0,\ \ 3,\ \ 0,\ \ 1$ &$\to4\to\widetilde{4}\to5\to2\to\widetilde{6}\to3$\\
$3,\ \ 1,\ \ 3,\ \ 2,\ \ 3,\ \ 0,\ \ 3$ &\\
$3,\ \ 2,\ \ 3,\ \ 1,\ \ 3,\ \ 1,\ \ 2$&\\
$3,\ \ 3,\ \ 3,\ \ 0,\ \ 3,\ \ 2,\ \ 1$ &\\
\hline
$1,\ \ 0,\ \ 3,\ \ 1,\ \ 1,\ \ 1,\ \ 0$ & \\
$1,\ \ 1,\ \ 3,\ \ 0,\ \ 1,\ \ 2,\ \ 3$ &\\
$3,\ \ 0,\ \ 1,\ \ 1,\ \ 3,\ \ 3,\ \ 0$ &\\
$3,\ \ 0,\ \ 3,\ \ 3,\ \ 3,\ \ 3,\ \ 2$ & $0\to\widetilde{0}\to1\to2\to\widetilde{6}\to7$\\
$3,\ \ 1,\ \ 1,\ \ 0,\ \ 3,\ \ 0,\ \ 3$ &$\to4\to\widetilde{4}\to5\to6\to\widetilde{2}\to3$\\
$3,\ \ 1,\ \ 3,\ \ 2,\ \ 3,\ \ 0,\ \ 1$ &\\
$3,\ \ 2,\ \ 3,\ \ 1,\ \ 3,\ \ 1,\ \ 0$&\\
$3,\ \ 3,\ \ 3,\ \ 0,\ \ 3,\ \ 2,\ \ 3$&\\
\hline
$1,\ \ 0,\ \ 3,\ \ 1,\ \ 3,\ \ 1,\ \ 2$ & \\
$1,\ \ 1,\ \ 3,\ \ 0,\ \ 3,\ \ 2,\ \ 1$ &\\
$3,\ \ 0,\ \ 1,\ \ 1,\ \ 1,\ \ 3,\ \ 2$ &\\
$3,\ \ 0,\ \ 3,\ \ 3,\ \ 1,\ \ 3,\ \ 0$ & $0\to\widetilde{0}\to1\to2\to\widetilde{2}\to7$\\
$3,\ \ 1,\ \ 1,\ \ 0,\ \ 1,\ \ 0,\ \ 1$ &$\to4\to\widetilde{4}\to5\to6\to\widetilde{6}\to3$\\
$3,\ \ 1,\ \ 3,\ \ 2,\ \ 1,\ \ 0,\ \ 3$ &\\
$3,\ \ 2,\ \ 3,\ \ 1,\ \ 1,\ \ 1,\ \ 2$&\\
$3,\ \ 3,\ \ 3,\ \ 0,\ \ 1,\ \ 2,\ \ 1$&\\
\hline
\end{tabular}
\\
\\
%
%
%
%

\medskip

\section{Some examples}\label{exam}

\begin{example}
Let $p=3$ and $\phi(x)=-\frac{2z^2+2z+1}{z^3-3z^2+z+1}$. The  dynamical system $(\P,\phi)$ is minimal.
\end{example}
The reduction $\!\!\!\mod 3$ of $\phi$ is
$$\overline{\phi}(z)=\frac{z^2+z+2}{z^3+z+1}.$$ So $\phi$ has good reduction. Consider the map $\overline{\phi}: \mathbb{P}^{1}(\mathbb{F}_3) \to \mathbb{P}^{1}(\mathbb{F}_3).$ By a simple calculation, we get 
$\overline{\phi}(1)= \infty, \ \overline{\phi}(\infty)=0, \ \overline{\phi}(0)=2, \ \overline{\phi}(2)=1$.
Thus $\overline{\phi}$ is transitive on $\mathbb{P}^{1}(\mathbb{F}_3)$. 
Further, using the software Mathmatica, we can check
$(\phi^{ 4})^{\prime}(1))\equiv (\phi^{ 4})^{\prime}(0)\equiv 1\ (\!\!\!\! \mod 3)$. Let $\sigma=(1,\infty,0,-1)$ be the cycle at level $1$.  Since $\phi^{ 4}(1)\equiv 7 \ (\!\!\!\!\mod 3^{3})$,  $\sigma$ grows. Let $\hat{\sigma}$ be the lift of
$\sigma$ at level $2$. By calculating $\phi^{12}(1)\equiv 10\  (\!\!\!\!\mod 3^{3})$, we deduce that $\hat{\sigma}$ grows which then implies that $\hat\sigma$ grows forever. Therefore, the system $(\P,\phi)$ is minimal.
\begin{example}
Let $p=3$ and $\phi(z)=\frac{2z+3}{(z-1)(z-2)}$.  The dynamical system $(\P,\phi)$ is not minimal and we decompose $\P$ as
  $$\P= B_1(0)\bigsqcup \left(\P\setminus B_1(0) \right),$$
  where $B_1(0)$ is a minimal component of $\phi$ and the points in $\P\setminus B_1(0)$ are attracted into
  $B_1(0)$.
\end{example}
One can check  that
$\overline{\phi}(z)=\frac{2z}{z^2+2}$. Then $\deg \overline{\phi}=\deg \phi$, which implies that $\phi$ has good reduction.
First, we consider the map $\overline{\phi}: \mathbb{P}^{1}(\mathbb{F}_3) \to \mathbb{P}^{1}(\mathbb{F}_3)$.
Since $\overline{\phi}(1)=\infty, \overline{\phi}(2)=\infty , \overline{\phi}(\infty)=0$ and $ \overline{\phi}(0)=0$, it follows that
 $$\phi(B_1(1)))\subset B_1(\infty),
\phi(B_1(2))\subset B_1(\infty),$$
$$\phi(B_1(\infty))\subset B_1(0) \mbox{~~and~~} \phi( B_1(0))\subset B_1(0).$$
So $\P\setminus B_1(0)$ lies in the attracting basin of $B_1(0)$.
 It suffices to study  the subsystem $(B_1(0), \phi)$.  We shall show that the system  $(B_1(0), \phi)$  is minimal.
In fact, the  derivative of $\phi$ at $0$ is $13/4$. Thus $\phi^{\prime}(0)\equiv 1 \ (\!\!\!\!\mod 3)$. Let $\sigma=(0)$ be the cycle at level $1$.
 Then $\sigma$ grows or splits. A simple  calculation shows  that
 $$\phi(0)\equiv 15 \ (\rm{mod} \ \  3^{3}), \ \phi(15)\equiv 3 \ (\rm{mod} \ \ 3^{3}),\ \phi(3)\equiv 18 \ (\rm{mod} \ \  3^{3}).$$
 So $\sigma$ grows, by Proposition \ref{minimalphi}, the lift of $\sigma$ grows too. It then follows that $\sigma$ grows forever.
 Hence, the system $(B_1(0), \phi)$ is minimal.


\bibliographystyle{plain}

\end{document}